\documentclass[reqno, 11pt]{amsart}

\usepackage[letterpaper,hmargin=1in,vmargin=1.2in]{geometry}
\usepackage{amsmath, amssymb, amsthm, verbatim, url}
\usepackage{graphicx}
\usepackage{enumerate, pinlabel}
\usepackage{amsmath,amssymb,amsthm,mathrsfs,graphicx,url}
\usepackage[usenames,dvipsnames]{color}
\usepackage[colorlinks=true,linkcolor=Red,citecolor=Green]{hyperref}

\DeclareMathOperator \im {Im}

\DeclareMathOperator \supp {supp}

\DeclareMathOperator \Ai {Ai}
\DeclareMathOperator \Bi {Bi}
\newtheorem{prop}{Proposition}
\newtheorem{lem}[prop]{Lemma}

\newtheorem{thm}{Theorem}
\newtheorem{que}{Question}
\numberwithin{equation}{section}

\author{Kiril Datchev}
\address{Department of Mathematics, Purdue University,  West Lafayette, IN, USA}
\email{kdatchev@purdue.edu}
\author{Long Jin}
\address{Yau Mathematical Sciences Center, Tsinghua University, Beijing, China}
\email{jinlong@tsinghua.edu.cn}

\title{Exponential lower resolvent bounds far away from trapped sets}

\thanks{The authors are very grateful to Semyon Dyatlov, Jeffrey Galkowski, Oran Gannot,  Jason Metcalfe, Vesselin Petkov, Plamen Stefanov, and Maciej Zworski for helpful discussions. Thanks also to the anonymous referees for their comments and suggestions. KD was supported by the Simons Foundation  through the Collaboration Grants for Mathematicians program, and by the National Science Foundation through grant DMS-1708511.}

\date{July 6, 2018}

\begin{document}

\begin{abstract}
 We give examples of semiclassical Schr\"odinger operators with exponentially large cutoff resolvent norms, even when the supports of the cutoff and potential are very far apart. The examples are radial, which allows us to analyze the resolvent kernel in detail using ordinary differential equation techniques. In particular, we identify a threshold spatial radius   where the  resolvent behavior changes.  We  apply these results to wave equations with radial wavespeed, identifying a corresponding threshold radius at which wave decay properties change.
\end{abstract}

\maketitle

\section{Introduction}

In the first part of this paper we study semiclassical resolvent estimates, and in the second part we apply the results to wave decay and non-decay.

\subsection{Semiclassical resolvent estimates}\label{s:semiintro} In this paper we investigate resolvent estimates for the semiclassical Schr\"odinger operator
\[P = P(h) = - h^2 \Delta + V(x),\] 
 where $V \colon \mathbb R^n \to \mathbb R$, $n \ge 2$. Initially we suppose $V \in C_c^\infty(\mathbb R^n)$, but later we will relax this condition.

\begin{que}\label{q:1}
For which $E_0>0$, and for which bounded and open sets $U \subset \mathbb R^n$, can we find an interval $I$ containing $E_0$ such that the incoming and outgoing cutoff resolvents obey
\begin{equation}\label{e:non}
\sup_{E \in I} \|{\bf 1}_U(P-E \pm i 0)^{-1}{\bf 1}_U\|_{L^2(\mathbb R^n) \to L^2(\mathbb R^n)} \le C/h,
\end{equation}
for all $h>0$ sufficiently small?
\end{que}

It is well known that the answer depends upon dynamical properties of the classical flow $\Phi(t) = \exp t(2\xi \partial_x - \partial_x V(x)\partial_\xi)$ in $T^*\mathbb R^n$. Of particular importance is the trapped set at energy $E_0$, which we denote $\mathcal K(E_0)$; this is the set of  $(x,\xi) \in T^*\mathbb R^n$ such that $|\xi|^2 + V(x) = E_0$ and $|\Phi(t)(x,\xi)|$ is bounded as $|t| \to \infty$.

If $\mathcal K(E_0)$ is empty, that is to say if $E_0$ is \textit{nontrapping}, then Robert and Tamura \cite{robtam}  show that the answer to Question~\ref{q:1} is that $U$ can be arbitrary.
Analogous results hold in much more general nontrapping situations; see e.g. \cite{v, hz} for some recent results, and see those papers and also \cite{bbr, zsurvey} for some pointers to the substantial literature on this topic.

But if $\mathcal K(E_0)$ is \textit{not} empty, then Bony, Burq, and Ramond \cite{bbr} show that there is $U$ such that for any interval $I$ containing $E_0$ we have
\begin{equation}\label{e:bbr}
 \sup_{E \in I}  \|{\bf 1}_U(P-E \pm i 0)^{-1}{\bf 1}_U\|_{L^2(\mathbb R^n) \to L^2(\mathbb R^n)}  \ge \frac{\log(1/h)}{Ch};
\end{equation}
more specifically it is enough if $T^*U$ contains an integral curve in $\mathcal K(E_0)$. Moreover, as we discuss below, the right hand side can sometimes be replaced by $e^{C/h}$.

Nevertheless, regardless of any trapping, for all $I \Subset (0,\infty)$ there exists $r_b>0$  such that
\eqref{e:non} holds whenever $U$ is disjoint from $B(0,r_b)$. 
Thus, if the distance between $T^*U$ and $\mathcal K(E_0)$ is large enough, then all losses due to trapping are removed. This was first shown by Cardoso and Vodev \cite{cv}, refining earlier work of Burq \cite{burq}, and analogous results  hold for much more general operators \cite{cv,rt,v,da,ddh,sh}.

It is not always necessary to cut off so far away: in \cite{dv} it is shown that if trapping is sufficiently mild, then we have \eqref{e:non} whenever $T^*U$ is disjoint from $\mathcal K(E_0)$. (By `trapping is sufficiently mild' we mean that the resolvent is polynomially bounded in $h^{-1}$; see \cite{dv} and also the survey \cite[\S3.2]{zsurvey} and the book \cite[Chapter 6]{dz} for more details, including sufficient conditions on $\mathcal K(E_0)$, and for references to some of the many known results of this kind.) Moreover, in that case \eqref{e:non} still holds if we replace ${\bf 1}_U$ by a microlocal cutoff vanishing only in a small neighborhood of $\mathcal K(E_0)$ in $T^*\mathbb R^n$. If $\mathcal K(E_0)$ is normally hyperbolic, then the vanishing hypothesis can be weakened further: see \cite{hv}. Propagation estimates play an important role in such results, and the connection between propagation estimates and polynomial resolvent bounds has been recently studied in \cite{bfrz}.

Our main result is that, when trapping is not mild, the situation can be dramatically different. Namely, losses due to trapping can show up very far away from the support of $V$:

\begin{thm} \label{t:first}
Suppose that $V \in C_c^\infty(B(0,1))$ is radial, $n \ge 2$, and $\min V <0$.  Let $R>1$ and let $U$ be a neighborhood of the sphere $\partial B(0,R)$: see Figure~\ref{f:bullseye}.  
There is $C > 0$ such that if
\begin{equation}\label{e:e0u}
0 < E_0  \le 1/CR^2, 
\end{equation}
then
\begin{equation}\label{e:texp}
 \sup_{E \in [E_0 - Ch, E_0 + Ch]}\| {\bf 1}_U(P-E \pm i0)^{-1} {\bf 1}_U  \|_{L^2(\mathbb R^n) \to L^2(\mathbb R^n)} \ge e^{1/Ch},
\end{equation}
for all $h>0$ small enough.
\end{thm}

\begin{figure}[h]
\labellist
\small
\pinlabel $\supp\!\,V$ [l] at 98 130
\pinlabel $U$ [l] at 120 17
\endlabellist
 \includegraphics[width=5cm]{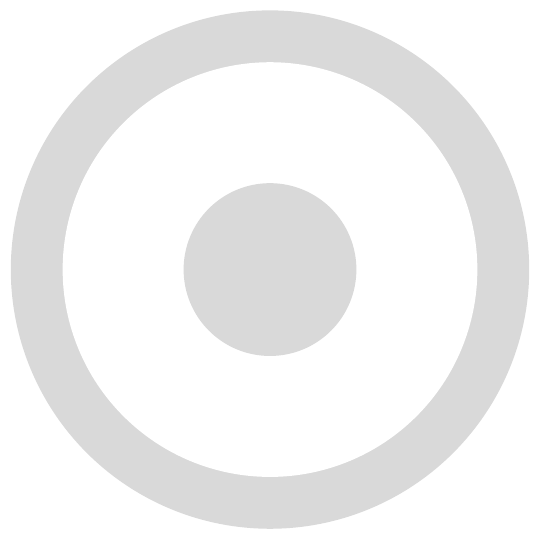}
 \caption{The relative positions of $U$ and $\supp V$.}\label{f:bullseye}
\end{figure}

The key point is that  the distance between $U$ and $\supp V$ (and hence, in particular, between $T^*U$ and the trapped set $\mathcal K(E_0)$) can be arbitrarily large. This seems to be a new phenomenon.

Moreover, the conditions on $U$ and $V$ can be weakened, and in this nice radial setting we can say  more. 
See Theorem \ref{t:second} below for a stronger and more general result, and \eqref{e:r2asy} for a more  precise version of \eqref{e:e0u}.

Lower bounds of the form \eqref{e:texp} with $\mathcal K(E_0) \subset T^*U$ stem from $O(e^{-1/Ch})$ quasimodes, and they are well known to hold in radial situations or under a barrier assumption, namely when $\{x \mid V(x)\le E_0\}$ has a bounded connected component contained in $U$. They  have been recently investigated in \cite{ddz}, where it is shown that, for $V$ satisfying a barrier assumption, $ {\bf 1}_U(P-E \pm i0)^{-1} {\bf 1}_U$ can be replaced by $ {\bf 1}_{U_L}(P-E \pm i0)^{-1} {\bf 1}_{U_R}$, with only one of $U_L$ and $U_R$ containing $\supp V$. Lower bounds for the  continuation of $ {\bf 1}_U(P-z)^{-1} {\bf 1}_U$ as $z$ crosses the positive real axis were recently studied in \cite{bp, dw}.

In the setting of Theroem \ref{t:first}, it is clear that no quasimodes can concentrate in $U$ because of the flow invariance of support of semiclassical measures. The quasimodes we use are concentrated in the set where $V$ is negative, and we show below (in Lemma \ref{l:lower}) that their exponential decay away from this set is very slow. This is a key difference between our result and  previous lower bounds as in \cite{ddz}.

The form of the right hand side of \eqref{e:texp} is essentially optimal:  for any $I \Subset(0,\infty)$ there is $C'>0$ such that for any $U$ we have
\begin{equation}
  \sup_{E \in I}\| {\bf 1}_U(P-E \pm i0)^{-1} {\bf 1}_U  \|_{L^2(\mathbb R^n) \to L^2(\mathbb R^n)} \le Ce^{C'/h}.
\end{equation}
This was first shown by Burq \cite{burq0,burq} in a more general setting, and has been generalized still further in \cite{cv,rt,v,da,ddh,sh, gannot}.

We can define the value $R_* = R_*(E_0; V)$ for more general potentials $V$ by
\begin{equation}\label{e:rstar}
 R_* : =\inf \{r>0 \mid \eqref{e:non} \textrm{ holds for some }I\textrm{ containing }E_0\textrm{ whenever } U \cap B(0,r) = \varnothing\}.
\end{equation}
As we mentioned previously, in \cite{cv, rt, v, da,ddh, sh} it is shown that $R_*$ is finite for quite general $V$ and $E_0>0$. In Theorem~\ref{t:first} we show that $R_* \to +\infty$ as $E_0 \to 0$ for a family of examples, and in \eqref{e:r2asy} we show that $R_* \sim E_0^{-1/2}$. Let us mention here that the finiteness of $R_*$, in this setting as well as in more general ones, has been applied to a variety of problems in scattering theory: see e.g. \cite{stef, michel, ghs, chr}. There are also well-known consequences for Schr\"odinger and wave evolution: see \S\ref{s:waveintro} below. These applications motivate the following question.

\begin{que}
 What can we say about the value of $R_*$ for more general $V$ and $E_0$?
\end{que}

For example, it would be interesting to know if the hypothesis that $V$ is radial in Theorem~\ref{t:first} could be weakened. 
The lower bounds on resonance widths of \cite{dm} make it seem unlikely that the radiality hypothesis could be removed altogether.

\subsection{Wave decay and non-decay}\label{s:waveintro}

We now give an application of the above results to decay and non-decay estimates for solutions to the wave equation with radial wavespeed, for simplicity restricting attention\footnote{See \S\ref{s:not} for a discussion of why we do this.} to the case $n \ge 3$. Thus, let $c \in C^\infty(\mathbb R^n;(0,\infty))$ be radial, so that $c(x) = c_0(r)$, and suppose
\begin{equation}\label{e:crhoc0}
 r \ge \rho \quad  \Longrightarrow \quad c_0(r) = \kappa,
\end{equation}
for some positive constants $\rho$ and $\kappa$. 
Given initial conditions $w_0, \, w_1 \in C_c^\infty(\mathbb R^n)$, let $w \in C^\infty(\mathbb R^n \times \mathbb R)$ be the solution to
\[
(\partial_t^2 - c(x)^2 \Delta)w(x,t) = 0, \qquad w(x,0) = w_0, \quad \partial_t w(x,0) = w_1.
\]
Then we have conservation of energy in $\mathbb R^n$:
\begin{equation}\label{e:econs}
\begin{split}
\mathcal E = \mathcal E[c, w_0,w_1] &= \int_{\mathbb R^n} |\nabla w_0(x)|^2 + |c(x)^{-1} w_1(x)|^2 dx \\
&= \int_{\mathbb R^n} |\nabla w(x,t)|^2 + |c(x)^{-1}\partial_t w(x,t)|^2 dx  , \qquad \forall t \in \mathbb R.
\end{split}\end{equation}

Our main result in this setting concerns  energy on a bounded open set $ U \Subset \mathbb R^n$:
\begin{equation}\label{e:eudef}
 \mathcal E_{U}(t) = \mathcal E_{ U}[c, w_0,w_1](t) = \int_{ U} |\nabla w(x,t)|^2 + |c(x)^{-1}\partial_t w(x,t)|^2 dx.
\end{equation}
This energy decays logarithmically in the sense that for all $ U \Subset \mathbb R^n$ and $k \in \mathbb N$ there is $C>0$ such that
\begin{equation}\label{e:log}
\mathcal E_{U}(t) \le C \langle \log t \rangle^{-k}, \qquad \forall t \in \mathbb R.
\end{equation}
In fact, these results are very robust. Proving the conservation of energy \eqref{e:econs} is  simple: one differentiates with respect to $t$ and integrates by parts, and so \eqref{e:econs} clearly holds for very general symmetric operators. Proving the logarithmic decay \eqref{e:log} is more complicated, but it too has been established in great generality. The study of wave decay has a long history and we do not attempt to survey it here. We just mention that the first general logarithmic decay results are due to Burq \cite{burq0}, and refer the reader also to \cite{burq, bouclet, moschidis, gannot, sh2} for more recent results on logarithmic decay and for more references.

We now bring in an assumption which ensures (stable) trapping: namely we assume that 
\begin{equation}\label{e:minl}
\min (r/c_0(r))' <0.
\end{equation}
Such a situation was considered by Ralston \cite{ral}, who showed that then there are sequences of resonances converging exponentially quickly to the real axis, and it is well-known that in particluar this means \eqref{e:log} cannot be improved if $T^*U$ contains the trapped set (see  \cite[\S7]{hosm} for a recent version of such a result in the setting of general relativity). Note also that if instead we had $\min (r/c_0(r))' >0$ then the problem would be nontrapping (see  \cite[p. 571]{ral}.)


To state our result we will also need the threshold radius
\begin{equation}\label{e:rcdef}
 R_c = \kappa \max_{r \in [0,\rho]} r/c_0(r),
\end{equation}
and we assume that
\begin{equation}\label{e:rcrho}
R_c>\rho.
\end{equation}

\begin{thm}\label{t:wave}
 Fix $c$ satisfying \eqref{e:crhoc0} and \eqref{e:rcrho}.
  \begin{enumerate}
  \item If $U \Subset \mathbb R^n$ is disjoint from the closed ball $\overline{B(0,R_c)}$, then there is $C>0$ such that
\begin{equation}\label{e:twdecay}
\int_{-\infty}^\infty \mathcal E_U(t) dt \le C \mathcal E,
\end{equation}
for all $w_0, \, w_1 \in C_c^\infty(\mathbb R^n)$.
\item If $U \Subset \mathbb R^n$  contains the sphere $\partial B(0,R)$ for some $R \in [\rho,R_c]$, then there is no $C>0$ such that \eqref{e:twdecay} holds for all $w_0, \, w_1 \in C_c^\infty(\mathbb R^n)$.
 \end{enumerate}
\end{thm}

\noindent\textbf{Remarks:}

\begin{enumerate}
 \item  One can check that \eqref{e:rcrho} implies \eqref{e:minl}. We make the stronger assumption \eqref{e:rcrho} because it simplifies our work, and because the most interesting examples have $R_c \gg \rho$.
 \item It is easy to construct families of examples such that $R_c \to +\infty$ with $\rho$ and $\kappa$ fixed. One way is to take $\psi \in C_c^\infty([0,1);[0,1])$ such that $\psi \equiv 1$ near $0$ and put
\[
c_0(r) = 1 - s \psi(r),
\]
with $s \in (0,1)$. Then $R_c \to \infty$ as $s \to 1$.
\item We see that $R_c$ is a threshold at which wave decay behavior changes, just as in Theorem \ref{t:second} we see that $r_2$ is a threshold at which resolvent norm behavior changes. Actually, by setting $E_0 = \kappa^{-2}$ and $V_0 = \kappa^{-2} - c_0^{-2}$ we have $R_c=r_2$ (see \S\ref{s:proof2} for more), and so $R_c$ is also a threshold for the behavior of the resolvent  $(-c^2 \Delta - \lambda^2)^{-1}$: see Lemma \ref{l:bres}.
\item We expect that the second part of Theorem \ref{t:wave} can be strengthened to take into account a possible loss of derivatives as follows: if $U \Subset \mathbb R^n$  contains the sphere $\partial B(0,R)$ for some $R \in [\rho,R_c]$, then for any $N \in \mathbb N$, there is no $C>0$ such that 
\[
 \int_{-\infty}^\infty \mathcal E_U(t) dt \le C \left(\|w_0\|^2_{H^N(\mathbb R^n)} + \|w_1\|^2_{H^{N-1}(\mathbb R^n)}\right) ,
\]
holds for all $w_0, \, w_1 \in C_c^\infty(\mathbb R^n)$. (Thanks to Jason Metcalfe for suggesting this comment.)
\end{enumerate}

We can interpret \eqref{e:twdecay} as an `exterior' wave decay estimate. Many variations of such estimates, including different types of smoothing and Strichartz estimates, have been established. See \cite{botz, mmt, mmtt, bgh, bouclets, miz, cw, cm, rt, mst, bcmp} and references therein for results where behavior away from some compact set is better than behavior in sets which overlap trapping. Our result seems to give the first examples of `bad' behavior extending arbitrarily far from the trapped set.

\subsection{Outline of the rest of the paper} In \S\ref{s:main} we state the main resolvent estimates of the paper, Theorem \ref{t:second}. In \S\ref{s:ode} we prove pointwise resolvent kernel bounds for a family of semiclassical ordinary differential operators, approximating the solutions by Airy functions using the remainder bounds of Olver \cite{Olver:Asymptotics}.  In \S\ref{s:proof2} we prove Theorem~\ref{t:second} and use it to prove Theorems \ref{t:first} and \ref{t:wave}.

\subsection{Notation}\label{s:not} In this paper $\Delta \le 0$ is the Euclidean Laplacian on $\mathbb R^n$ for some $n \ge 2$, $h>0$ is a (small) semiclassical parameter, ${\bf 1}_U$ is the characteristic function of $U$, $C>0$ is a constant which may change from line to line, $A \Subset B$ means that the closure of $A$ is a compact subset of $B$, $B(a,b)$ is the  ball with center $a$ and radius $b$, the sphere $\partial B(a,b)$ is its boundary, $r=|x|$ is the radial coordinate in $\mathbb R^n$, $\langle t \rangle = (1+t^2)^{1/2}$, and $\sum_\pm Q(\pm) = Q(+) + Q(-)$. $\Ai$ and $\Bi$ are Airy functions, see Appendix \ref{s:airy}.

The radius $\rho$ and wavespeed $\kappa$ are defined in \eqref{e:crhoc0}, and the radius $R_c$ is defined in \eqref{e:rcdef}. The potentials $V$, $V_0$, and $V_m$ are defined in \eqref{e:vmdef} and the preceding sentences. The angular momentum $M_0$ and the radii $r_2$ and $r_1$ are defined in terms of the potential $V$ and the energy level $E_0$ in \eqref{e:m0def}, \eqref{e:r2def}, and \eqref{e:r1def} respectively: see also Figure \ref{f:v0} and Lemma \ref{l:phi} for more on these important quantities. The Schr\"odinger operator $P_m$ is defined in \eqref{e:radop}, its domain $\mathcal D$ is defined in terms of the boundary condition $\mathcal B$ immediately afterwards, and its resolvent kernel $K(r,r')$ is then given in \eqref{e:kdef}. We also sometimes use the domain $\mathcal D_{r_2}$ given in \eqref{e:dr2def}. The angular momenta $m_j$ are defined in terms of the spherical eigenvalues $\sigma_j$ in \eqref{e:mjdef}.

In \S\ref{s:wave} we use the homogeneous Sobolev space $\dot H^1(\mathbb R^n)$, defined to be the completion of $C_c^\infty(\mathbb R^n)$ with respect to the norm $u \mapsto \|\nabla u\|_{L^2(\mathbb R^n)}$. When $n =2$ this is not a space of distributions and various technical difficulties arise (e.g. multiplication by a function in $C_c^\infty(\mathbb R^n)$ is not a bounded operator). For simplicity, in  \S\ref{s:waveintro} and \S\ref{s:wave} we stick to the case $n \ge 3$ (so that, in particular, Sobolev embedding implies $ \dot H^1(\mathbb R^n) \subset L^{\frac {2n}{n-2}}(\mathbb R^n)$), but see \cite{sh2} for methods which cover the case $n=2$. Note that in Theorems \ref{t:first} and \ref{t:second} these difficulties do not appear and we allow $n \ge 2$.

\section{Main theorem}\label{s:main}

In this section we state our main semiclassical resolvent estimates. We begin with the assumptions, which are weaker but more complicated than the ones for Theorem \ref{t:first}. 

Let $n \ge 2$, and let $V \colon\mathbb R^n\setminus\{0\} \to \mathbb R$ be radial, so that $V(x) = V_0(r)$, and suppose\footnote{To keep things  simpler while still capturing  the interesting phenomena, one can restrict attention to the case $V_0(r) \in C^\infty_c([0,\infty))$.}  $V_0(r) \in r^{-2}C^3([0,\infty))$ is bounded below and  $\left|r^{2+k}\partial_r^{k} V_0(r)\right|$ is bounded for all
 $r >0$ and $k \in \{0, \, 1, \, 2, \, 3\}$. 

Then $P= -h^2 \Delta +V$ is selfadjoint on a domain containing $C_c^\infty(\mathbb R^n \setminus \{0\})$, and we fix such a domain. For $E>0$ we can define and study the incoming and outgoing resolvents $(P-E\pm i0)^{-1}$ using separation of variables as we recall in \S\ref{s:proof2} below.

For $m \in \mathbb R$, let
\begin{equation}\label{e:vmdef}
 V_m(r) = V_0(r) + m r^{-2}.
\end{equation}
This is the effective potential which arises when we write the Laplacian in polar coordinates; we think of $m$ as the angular momentum.
For $E_0 > 0$, put
\begin{equation}\label{e:m0def}
M_0 = M_0(E_0) = \sup\{m\ge0 \mid V_m^{-1}(E_0) \textrm{ has at least two points}\},
\end{equation}
and suppose 
\begin{equation}\label{e:mass}  M_0 \textrm{ is finite.}\end{equation}
The trapping we use occurs at the angular momentum $M_0$, and we will also need the following two radii.
Put
\begin{equation}\label{e:r2def}
r_2 = r_2(E_0) =  \max V_{M_0}^{-1}(E_0),
\end{equation}
and suppose 
\begin{equation}\label{e:monass} V'_{M_0}(r) < 0 \textrm{ for all }r \ge r_2.\end{equation}
Put
\begin{equation}\label{e:r1def}
 r_1 = r_1(E_0) =  \max \left(V_{M_0}^{-1}(E_0)\setminus\{r_2(E_0)\}\right).
\end{equation}
Note that if $V_0$ is compactly supported and $\min V_0<0$, then the assumptions \eqref{e:mass} and \eqref{e:monass} are automatically satisfied for $E_0>0$ sufficiently small; we also have more explicit formulas for $M_0$, $r_1$, and $r_2$, which we derive in \S\ref{s:4.2} below. These assumptions imply that $V_{M_0}'(r_1) = 0$, so that the trapped set $\mathcal K(E_0)$ contains circular orbits in $T^*\partial B(0,r_1)$, and these are the trapped orbits that we will use (in Lemma \ref{l:lower}) to prove exponential lower bounds. 
See Figure \ref{f:v0}.

\begin{figure}[h]
\labellist
\small
\pinlabel $E_0$ [l] at 5 75
\pinlabel $r_1$ [l] at 72 16
\pinlabel $r_2$ [l] at 261 16
\endlabellist
 \includegraphics[width=10cm]{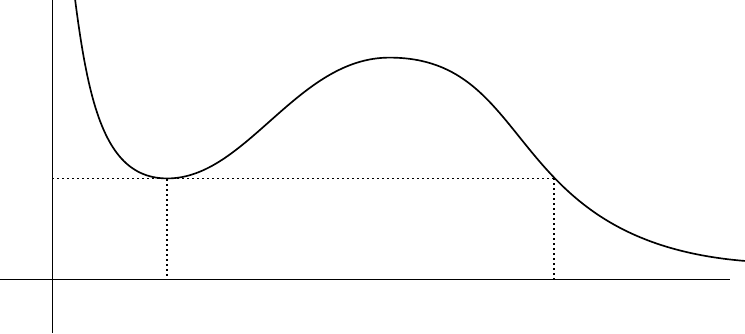}
 \caption{A possible graph of $V_{M_0}$.}\label{f:v0}
\end{figure}

\begin{thm} \label{t:second}
Let $V$, $E_0$, $r_1$, and $r_2$ be as above. 

If $U \subset \mathbb R^n$ is bounded, open, and disjoint from a neighborhood of $\overline{B(0,r_2)}$, then there is an interval $I$ containing $E_0$ and a constant $C_0$ such that
\begin{equation}\label{e:tnontr}
 \sup_{E \in I}\|  {\bf 1}_U(P-E \pm i0)^{-1} {\bf 1}_U  \|_{L^2(\mathbb R^n) \to L^2(\mathbb R^n)} \le C_0/h,
\end{equation}
for all $h>0$ small enough.

On the other hand, let $U_L$ and $U_R$ be bounded open sets in $\mathbb R^n$ containing spheres $\partial B(0,r_L)$ and $\partial B(0,r_R)$ respectively, such that $\min(r_L,r_R) \in [r_1,r_2]$. Then there are constants $ C_1$ and $C_2$ such that
\begin{equation}\label{e:texpasy}
 \sup_{E \in [E_0 - C_1h, E_0 + C_1h]}\| {\bf 1}_{U_L}(P-E\pm i0)^{-1} {\bf 1}_{U_R}  \|_{L^2(\mathbb R^n) \to L^2(\mathbb R^n)} \ge e^{C_2/h},
\end{equation}
for all $h>0$ small enough. If in addition we have $\max V < E_0$, then
\begin{equation}\label{e:texpasyseq}
 \| {\bf 1}_{U_L}(P-E_0\pm i0)^{-1} {\bf 1}_{U_R}  \|_{L^2(\mathbb R^n) \to L^2(\mathbb R^n)} \ge e^{C_2/h},
\end{equation}
for $h$ tending to $0$ along a sequence of positive values.
\end{thm}

The main point is the value  $r_2 = r_2(E_0)$, which corresponds to $R_*$ from \eqref{e:rstar}. Comparing \eqref{e:tnontr} with \eqref{e:texpasy} and \eqref{e:texpasyseq} we see that $r_2$ is a threshold at which the behavior of the resolvent changes.

We now discuss the values of $C_1$ and $C_2$, which come from Lemma \ref{l:lower} below. The former is related to an eigenvalue of an interior problem, and the latter to an Agmon distance, both for the effective potential $V_{M_0}$. Such eigenvalues are known to approximate real parts of resonances near the real axis, while  Agmon distances correspond to imaginary parts of the same resonances: see  \cite{hs} (especially \S 11 of that paper) and \cite{flm} for results in a well-in-an-island setting, and \cite[Corollary, \S5]{nsz} for an abstract statement. Let us emphasize that our lower bounds are in terms of an Agmon distance for $V_{M_0}$ rather than for $V$, and the former may be much greater than the latter (for example, the latter vanishes if $\max V<E_0$). See also \cite{servat} for one-dimensional resonance asymptotics using methods in some ways similar to ours, and \cite{dm} for a more recent higher-dimensional result and more references.

An interesting special case is the one where $V$ has a unique local minimum, located at the origin, and moreover this minumum is nondegenerate and $V_0(0) = E_0>0$. Then $M_0(E_0) = 0$, so that $V(x) = V_{M_0}(r)$ is a well-in-an-island type potential, and $E_0$ is at the bottom of the well. In that case, by \cite[Proposition 4.1]{n}, there is $\mu>0$ such that for any $U$ we have
\begin{equation}
\sup_{E \in [E_0-\mu h,E_0 + \mu h]}\|  {\bf 1}_U(P-E)^{-1} {\bf 1}_U  \|_{L^2(\mathbb R^n) \to L^2(\mathbb R^n)} \le C/h.
\end{equation}
This upper bound shows that the form of the interval $[E_0 - C_1h, E_0 + C_1h]$ in \eqref{e:texpasy} is optimal in general. See also \cite[\S1]{bbr} and \cite[\S 4]{ddz} for further discussion.

\section{Semiclassical ODE asymptotics}\label{s:ode}

In this section we prove pointwise resolvent estimates, for energies $E$ near $E_0$, for
\begin{equation}\label{e:radop}
P_m = P_m(h) =-h^2 \partial_r^2 + V_m(r),
\end{equation}
where $V_m$ is given by \eqref{e:vmdef} with $m\ge - h^2/4$. Let $\mathcal D = \mathcal D(m,h) \subset L^2(\mathbb R_+)$ be a domain for $P_m$ so that $P_m$ is selfadjoint and $C_c^\infty((0,\infty))  \subset \mathcal D$.

We briefly recall some facts about $\mathcal D$. By Proposition 2 and Theorems X.7, X.8, and X.10 of \cite[Appendix to X.I]{rs} we have $\mathcal D = \{u \in L^2(\mathbb R_+) \mid P_m u \in L^2(\mathbb R_+) \textrm{ and } \mathcal B u = 0\}$, where $\mathcal Bu$ is a boundary condition at $0$ with real coefficients, and moreover $\mathcal B = 0$ unless $m = O(h^2)$. If $m = O(h^2)$, then we may have $\mathcal B \ne 0$; below we will not need further information about $\mathcal B$, but  see \cite{rel, bg} and \cite[\S10.4]{zet} for descriptions of the possibilities. Note finally that $\mathcal D$ is preserved by complex conjugation because $P_m$ has real coefficients.

The outgoing resolvent kernel at energy $E>0$ is given by
\begin{equation}\label{e:kdef}
 K(r,r')  = K(r,r';m;E;h) = -u_0(r)u_1(r')/h^2W, \qquad \textrm{for } r \le r',
\end{equation}
and it obeys $K(r,r') = K(r',r)$, where $u_0$ and $u_1$ are certain solutions  to 
\begin{equation}\label{e:pjuj}
P_mu_j = E u_j,
\end{equation}
and
$W = u_0u_1' - u_0'u_1$ is their Wronskian. More specifically $u_0 \in L^2((0,1))$  and satisfies $\mathcal Bu_0 = 0$,  and $u_1(r)$ is outgoing, that is to say it is asymptotic to a multiple of $e^{i r \sqrt E/h}$ as $r \to \infty$. For convenience we assume without loss of generality that $u_0$ is real-valued.

In the remainder of \S\ref{s:ode} we prove three lemmas, each of which bounds $K(r,r')$  for a different range of $r$, $r'$, $m$, and $E$. The first two will be used to prove \eqref{e:tnontr}, and the third to prove \eqref{e:texpasy} and \eqref{e:texpasyseq}. 

In the first lemma $m$ is small enough that no turning point analysis is needed.

\begin{lem}\label{l:noturn}
Fix $r_{2+}>r_2$, $M>0$, and $I \Subset(0,\infty)$ such that $V_M(r)<E$ for all $r \ge r_{2+}$ and $E \in I$. Then
\begin{equation}\label{e:krr'bigbetter}
 |K(r,r')| \le \frac{h^{-1}(1 + O(h))}{(E-V_m(r))^{1/4}(E-V_m(r'))^{1/4}},
\end{equation}
uniformly for all $r \ge r_{2+}$, $r' \ge r_{2+}$, $E \in I$, and $m \in [0,M]$.
\end{lem}
In \S\ref{s:proof2} we will specify $r_{2+}$ and $M$; they will be slightly larger than $r_2$ and $M_0$ respectively.

Before giving the proof we give the idea. We will use the fact that $u_0$ and $u_1$ are each oscillatory, rather than exponentially growing or decaying, since we are in the classically allowed region $E > V_m$. So the upper bound follows from a lower bound on the Wronskian; this in turn follows from the fact that, roughly speaking, $u_1$ is like $\exp\left(\frac i h \int \sqrt{E-V_m}\right)$ since it is outgoing, while $u_0$ is equal amounts $\exp\left(\frac i h \int \sqrt{E-V_m}\right)$ and $\exp\left(-\frac i h \int \sqrt{E-V_m}\right)$  since it is real-valued.

\begin{proof}
For any $u_0$ as above, by \cite[\S 6.2.4]{Olver:Asymptotics} there are real constants $A = A (h)$ and $B = B(h)$ such that for $r \ge r_{2+}$ we have
\begin{equation}\label{e:u0lg}
  u_0(r) = (E-V_m(r))^{-1/4}\left(\sum_{\pm} (A \pm i B) \exp\left(\pm\frac i h \int_{r_{2+}}^r \sqrt{E-V_m(r')}dr'\right)(1+\varepsilon_\pm(r))\right),
\end{equation}
where $\varepsilon_+$ and $\varepsilon_-$ satisfy
\begin{equation}\label{e:epspm}
 |\varepsilon_\pm(r)| +  h|\varepsilon_\pm'(r)|   \le  C h r^{-1},
\end{equation}
when $r \ge r_{2+}$.

Again by \cite[\S 6.2.4]{Olver:Asymptotics}, we can normalize $u_1$ to be the outgoing  solution of \eqref{e:pjuj} given by
\begin{equation}\label{e:u1lg}
  u_1(r) = (E-V_m(r))^{-1/4}\exp\left(\frac i h \int_{r_{2+}}^r \sqrt{E-V_m(r')}dr'\right)(1+\varepsilon_+(r)),
\end{equation}
for $r \ge r_{2+}$.

We  compute the Wronskian 
\begin{equation}\label{e:wronu0u1}
 W = W(u_0,u_1) = \frac{A-iB}{\sqrt{E-V_m(r)}} \cdot \frac {2i}h \sqrt{E-V_m(r)} (1+O(hr^{-1})) = 2(B+iA)h^{-1},
\end{equation}
where we dropped the remainder because $W$ is independent of $r$.
Combining this with \eqref{e:kdef}, \eqref{e:u0lg}, and \eqref{e:u1lg} gives \eqref{e:krr'bigbetter}.
\end{proof}

In the second lemma $m$ is large enough that a turning point analysis is needed.
 For our purposes the following bound which blows up near the turning point is sufficient, even though $K$ is of course continuous there. We state a result for all $r>0$ and $r'>0$, even though we only use a smaller range in our application, since the result for the full range is obtained with no extra effort. We remark that a closely related turning point analysis appears in a recent paper of Yafaev on semiclassical asymptotics for eigenfunctions in a potential well \cite{y}.

\begin{lem}\label{l:oneturn}
Fix  $M>0$ such that $V'_M(r)<0$ for all $r \ge r_2$, and fix $I\Subset (0,\infty)$ containing $E_0$ such that $V_M(r)>E$ for all $r<r_2$ and $E \in I$. Then
\begin{equation}\label{e:kturn}
|K(r,r')| \le \frac{C_A \pi h^{-1}(1 + O(m^{-1/2}h))} {|E-V_m(r)|^{1/4}|E-V_m(r')|^{1/4}},
\end{equation}
uniformly for all $r>0$, $r'>0$, $m \ge M$ and $E \in I$, where $C_A$ is given by \eqref{e:abbound}.
\end{lem}

Before giving the proof we give the idea. By rescaling, we can use $m^{-1/2}h$ as a new semiclassical parameter, and the turning point $R$ is roughly given by $V_m^{-1}(E) \ge r_2$; the classically forbidden region is $r < R$ and classically allowed region is $r>R$. In the classically allowed region the bound holds for the same reason that it did in Lemma \ref{l:noturn}. In the classically forbidden region, the solutions are exponentially growing and decaying, rather than oscillatory, because $E<V_m$ there. But $u_0$ is forced to have only an exponentially decaying component and no exponentially growing one by the condition $u_0 \in L^2((0,1))$. Since we are estimating an expression of the form $u_0(r)u_1(r')$ with $r \le r'$, the decay from $u_0$  beats the growth from $u_1$. The Wronskian cannot be very small for the same reason as in Lemma \ref{l:noturn}. Near the turning point these arguments break down, as can be seen from the weakness of \eqref{e:kturn} when $r$ or $r'$ is close to $V_m^{-1}(E)$.

\begin{proof}  

The proof proceeds in four steps. In the first we introduce some useful notation, including a change of variable $r \mapsto \zeta(r)$ in the manner of \cite[\S11.3]{Olver:Asymptotics}. In the second we express $u_0$ in terms of Airy functions, and compute asymptotics as $m^{1/2}h^{-1}$ and $r$ become large. In the third we do the same for $u_1$, and in the fourth we compute the Wronskian and combine the previous results to conclude.

Put 
\begin{equation}
m' = m + \tfrac {h^2}4 \quad \textrm{ and }\quad  R = V_{m'}^{-1}(E),
\end{equation}
and note that 
\begin{equation}\label{e:sqrtm}
R + |V'_{m'}(R)|^{-1} \le C m^{1/2};
\end{equation}
see Figure \ref{f:v2}.

\begin{figure}[h]
\labellist
\small
\pinlabel $E$ [l] at 7 75
\pinlabel $r_2$ [l] at 211 14
\pinlabel $R$ [l] at 257 16
\endlabellist
 \includegraphics[width=10cm]{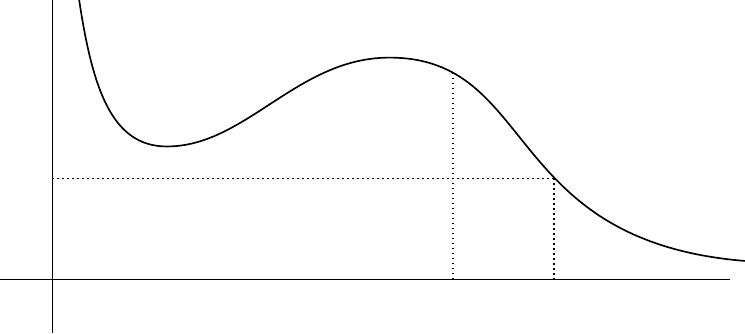}
 \caption{A possible graph of $V_{m'}$.}\label{f:v2}
\end{figure}

Following \cite[\S11.3.1]{Olver:Asymptotics}, we rewrite \eqref{e:pjuj} as
\begin{equation}\label{e:fgdef}
 u'' = \left((m^{1/2}h^{-1})^2 f + g\right)u, \qquad \textrm{where } \quad  f = m^{-1}(V_{m'} - E)\quad  \textrm{ and } \quad g = -\tfrac 14 r^{-2}.
\end{equation}
As we will see in \eqref{e:epsab-} below, this decomposition leads to good asymptotic properties as $r \to 0$.

Define an increasing bijection  $(0,\infty) \ni r \mapsto \zeta(r) = \zeta (r;m;h) \in \mathbb R$ by 
\begin{equation}\label{e:zetadef}
 \zeta(r) = \pm\left|\frac 3{2h} \int_R^r \sqrt{E-V_{m'}(r')}dr'\right|^{2/3},\quad  \textrm{ when } \pm(r-R)\ge0.
\end{equation}
Note that our $\zeta$ differs from the one used in \cite[\S11.3]{Olver:Asymptotics} by a factor of $m^{1/3}h^{-2/3}$.

We will need the following bound for $\zeta$: for $r>0$ sufficiently small we have
\begin{equation}\label{e:zetarsmall}
\frac 23 m^{-1/2}h(- \zeta(r))^{3/2} = m^{-1/2}\int_r^R\sqrt{V_{m'}(r') - E}dr' \ge \ln(1/r)/C.
\end{equation}

 By \cite[\S 11.3.3]{Olver:Asymptotics}, there are constants $A_0 = A_0(h)$ and $B_0 = B_0(h)$ such that
\begin{equation}\label{e:u0airy}
 u_0(r) = \left(\frac{\zeta(r)}{E-V_{m'}(r)}\right)^{1/4}\left( A_0\left(\Ai(-\zeta(r)) + \varepsilon_A(r)\right) + B_0 \left(\Bi(-\zeta(r)) + \varepsilon_B(r)\right)\right),
\end{equation}
where $\Ai$ and $\Bi$ are given by \eqref{e:airydef}, and  $\varepsilon_A$ and $\varepsilon_B$ satisfy
\begin{equation}\label{e:epsab+}
|\varepsilon_A(r)| + |\varepsilon_B(r)| \le  C  m^{-1/2} h \langle \zeta (r) \rangle^{-1/4},  \quad \textrm{ when } r \ge R,
\end{equation}
and
\begin{equation}\label{e:epsab-}
\frac{|\varepsilon_A(r)|}{\Ai(-\zeta(r))} + \frac{|\varepsilon_B(r)|}{\Bi(-\zeta(r))}   \le  C m^{-1/2} h \langle \zeta (r) \rangle^{-1/4} ,   \quad \textrm{ when } r \le R;
\end{equation}
see Appendix \ref{s:apperror} for more details. We use $m'$ rather than $m$ in our definition of $\zeta$ because this choice makes \eqref{e:epsab-} hold uniformly for all $r \in (0,R]$, rather than merely uniformly on compact subintervals of $(0,R]$: see Appendix \ref{s:apperror} and also \cite[\S11.4.1 and \S 6.4.3]{Olver:Asymptotics}.

Recalling that $u_0 \in L^2((0,1))$, and inserting \eqref{e:epsab-}, \eqref{e:ax}, and \eqref{e:bx} into \eqref{e:u0airy}, we see that $B_0=0$. Without loss of generality we  normalize $u_0$ so that $A_0 = 1$ and
\begin{equation}\label{e:u0ai}
  u_0(r) = \left(\frac{\zeta(r)}{E-V_{m'}(r)}\right)^{1/4}\left( \Ai(-\zeta(r)) + \varepsilon_A(r)\right).
\end{equation}

We now derive the simpler and better asymptotics which hold for large $r$; these will ease the computation of the Wronskian later. If $r\ge R+1$, then inserting \eqref{e:epsab+} and \eqref{e:a-x} into  \eqref{e:u0ai}  gives
\begin{equation}\label{e:u0cos}
 u_0(r) =  \pi^{-1/2} (E - V_{m'}(r))^{-1/4} \left(\cos\left(\frac 23 \zeta(r)^{3/2}- \frac \pi 4\right) + O(m^{-1/2}h) + O(\zeta(r)^{-3/2})\right).
\end{equation}
On the other hand, as in \eqref{e:u0lg}, there are real constants $A = A(h)$ and $B=B(h)$ such that for $r \ge R+1$ we have 
\begin{equation}\label{e:u0lgm}
  u_0(r) = (E-V_{m'}(r))^{-1/4}\left(\sum_{\pm} (A \pm i B) \exp\left(\pm\frac i h \int_{R}^r \sqrt{E-V_{m'}(r')}dr'\right)(1+\varepsilon_\pm(r))\right),
\end{equation}
where $\varepsilon_+$ and $\varepsilon_-$ satisfy
\begin{equation}\label{e:epspmm}
 |\varepsilon_\pm(r)| + m^{-1/2}h|\varepsilon_\pm'(r)|   \le  Cm^{-1/2} h r^{-1},
\end{equation}
when $r \ge R+1$. Setting \eqref{e:u0cos} equal to \eqref{e:u0lgm} and applying \eqref{e:epspmm} gives
\begin{equation}\label{e:u0connect}
 \pi^{-1/2} \cos\left(\frac 1 h\varphi(r) - \frac \pi 4\right) + O(m^{-1/2}h) = \sum_{\pm} (A \pm i B) \exp\left(\pm\frac i h \varphi(r)\right)(1+O(m^{-1/2}h)),
\end{equation}
for $r$ large enough (depending on $m$ and $h$), where we used
\[\varphi(r) =  \frac {2h} 3 \zeta(r)^{3/2} = \int_{R}^r \sqrt{E-V_{m'}(r')}dr' \to \infty \textrm{ as } r \to \infty.\] 
Hence
\begin{equation}\label{e:apmib}
 A \pm i B = 2^{-1}\pi^{-1/2}e^{\mp i\pi/4} + O(m^{-1/2}h).
\end{equation}

Now we turn to $u_1$. As in \eqref{e:u1lg}, we normalize it by setting
\begin{equation}\label{e:u1lgm}
  u_1(r) = (E-V_{m'}(r))^{-1/4}\exp\left(\frac i h \int_{R}^r \sqrt{E-V_{m'}(r')}dr'\right)(1+\varepsilon_+(r)),
\end{equation}
for $r \ge R+1$.
Using \cite[\S 11.3.3]{Olver:Asymptotics} again, there are constants $A_1 = A_1(h)$ and $B_1 = B_1(h)$ such that for $r \in (0,\infty)$ we have
\begin{equation}\label{e:u1airy}
  u_1(r) = \left(\frac{\zeta(r)}{E-V_{m'}(r)}\right)^{1/4}\left( A_1\left(\Ai(-\zeta(r)) + \varepsilon_A(r)\right) + B_1 \left(\Bi(-\zeta(r)) + \varepsilon_B(r)\right)\right).
\end{equation}
Proceeding as in the proof of \eqref{e:u0connect} gives
\[\begin{split}
 A_1\left(\cos\left(\frac 1 h \varphi(r)- \frac \pi 4\right) + O(m^{-1/2}h)\right) - B_1 \left(\sin\left(\frac 1 h \varphi(r)- \frac \pi 4\right) +  O(m^{-1/2}h)\right) \\ = \sqrt \pi\exp\left(\frac i h \varphi(r)\right)(1+O(m^{-1/2}hr^{-1})).
\end{split}\]
Hence we have
\begin{equation}\label{e:a1b1}
 A_1 = \sqrt{\pi}e^{i\pi/4} + O(m^{-1/2}h), \qquad B_1 = \sqrt{\pi}e^{-i\pi/4} + O(m^{-1/2}h),
\end{equation}
which implies
\begin{equation}\label{e:u1abs}
|u_1(r)|^2 = \pi \left|\frac{\zeta(r)}{E-V_{m'}(r)}\right|^{1/2}\left(\Ai(-\zeta(r))^2 + \Bi(-\zeta(r))^2\right)(1+O(m^{-1/2}h)).
\end{equation}
We compute the Wronskian $W = u_0u_1'-u_0'u_1$ as in the proof of \eqref{e:wronu0u1}, and apply \eqref{e:apmib}:
\begin{equation}\label{e:wturn}
 W =2(B+iA)h^{-1} =  \pi^{-1/2}e^{3i\pi/4}h^{-1}(1 + O(m^{-1/2}h)).
\end{equation}

Now \eqref{e:kturn} follows from inserting \eqref{e:u0ai}, \eqref{e:u1abs},  and  \eqref{e:wturn} into \eqref{e:kdef}, applying \eqref{e:epsab+}, \eqref{e:epsab-},  and \eqref{e:abbound}, and observing that $m' = m +h^2/4$ allows us to replace $V_{m'}$ by $V_m$ in the final statement (actually, keeping $V_{m'}$  gives a slightly better bound).
\end{proof}

In the third lemma the analysis is similar to that in the second, but slightly easier since we consider only a bounded range of $m$. To obtain good lower bounds we consider only a particular energy level, rather than an interval of energies as in the previous two lemmas. We actually obtain an asymptotic, rather than merely a lower bound, with no extra effort.

Having information about $K$ in \eqref{e:krr'midmid}, rather than just $|K|$, is important for our application to wave non-decay in Theorem \ref{t:wave}, where we need a lower bound on $|\im K|$.

\begin{lem}\label{l:lower}
Fix $r_{1+} < r_{2-} \in (r_1,r_2)$ and fix $r_{2+}>r_2$. For any $m = m(h) = M_0 + O(h)$, there is an energy level $E = E(h) = E_0 + O(h)$,  such that
\begin{equation}\label{e:krr'mid}
 |K(r,r')| =  \frac{e^{S(r)  /h}h^{-1}(1+O(h^{1/3}))}{2(V_m(r)-E)^{1/4}(E-V_m(r'))^{1/4}},  \textrm{ when } r \in [r_{1+}, r_{2-}], \ r' \ge r_{2+},
\end{equation}
and
\begin{equation}\label{e:krr'midmid}
 K(r,r') =  \frac{- e^{-i\pi/6}e^{S(r)  /h}e^{S(r')  /h}h^{-1}(1+O(h^{1/3}))}{2(V_m(r)-E)^{1/4}(V_m(r')-E)^{1/4}}, \textrm{ when } r, \ r' \in [r_{1+}, r_{2-}]
\end{equation}
 where
\begin{equation}
 S(r) = \int_r^{R}\sqrt{V_m(r') - E}dr',
\end{equation}
with $R = \max V_m^{-1}(E) = r_2 + O(h)$. See Figure \ref{f:v}.

Moreover, it suffices to take $E$ to be an eigenvalue of $P_m$ as an operator on $L^2((0,r_2))$ with domain \begin{equation}\label{e:dr2def}
\mathcal D_{r_2} = \{u \in L^2((0,r_2)) \mid P_m u \in  L^2((0,r_2)) \textrm{ and } \mathcal B u = u(r_2) = 0\}.                                                                                                        \end{equation} 
\end{lem}

\begin{figure}[h]
\labellist
\small
\pinlabel $E_0$ [l] at 5 75
\pinlabel $E$ [l] at 6 95
\pinlabel $r_1$ [l] at 72 16
\pinlabel $r_{1+}$ [l] at 125 16
\pinlabel $r_{2-}$ [l] at 210 16
\pinlabel $R$ [l] at 238 17
\pinlabel $r_2$ [l] at 261 16
\pinlabel $r_{2+}$ [l] at 281 16
\endlabellist
 \includegraphics[width=10cm]{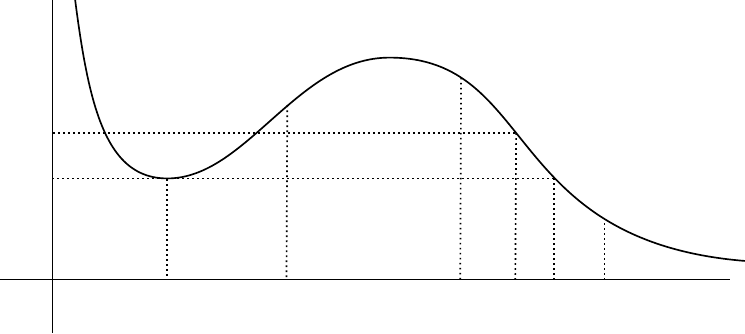}
 \caption{A possible graph of $V_{m}$ in the case $m=M_0$. We think of $r_{2+} - r_{2-}$ and $r_{1+} - r_1$ as being very small, although for clarity this is not so in the picture. In \S\ref{s:proof2} we will take $m = M_0 + O(h)$, and the graph must be suitably changed: in particular, in the proof of \eqref{e:texpasyseq} we will need $E=E_0$ and hence $m<M_0$.}\label{f:v}
\end{figure}
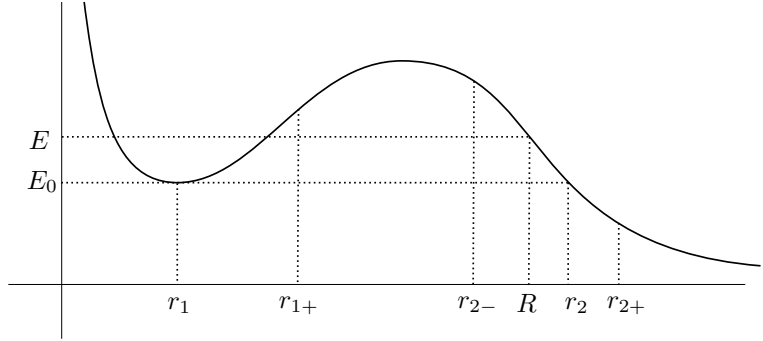

Before giving the proof we give the idea. The operator $P_m - E$ has the same turning point behavior here as in Lemma \ref{l:oneturn}, but this time we must take advantage of the trapping occuring at $r_1$. We do this by finding an energy level $E = E_0 + O(h)$ which is also an eigenvalue of an interior problem (we use the Dirichlet problem on $(0,r_2)$) and then taking as $u_0$ the corresponding eigenfunction (this is our quasimode). This way both $u_0$ and $u_1$ grow exponentially in the classically forbidden region between $r_1$ and $r_2$, giving the desired result.

\begin{proof}
We begin by proving that $P_m$ has an eigenvalue $E = E_0 + O(h)$ as an operator on $L^2((0,r_2))$ with domain $\mathcal D_{r_2}$. Note first that the spectrum of this operator is discrete since the domain is contained in $H^1((0,r_2))$, and let $E$  be the bottom of the spectrum. Then we have $E \ge E_0 - Ch$ because $E_0 = \min\{V_{M_0}(r) \mid r \in (0,r_2)\}$ by  \eqref{e:m0def}, the definition of $M_0$. To see that $E \le E_0 + Ch$, fix $\alpha>0$ such that 
\begin{equation}\label{e:vmalpha}
V_m(r) - E_0 \le \alpha^2(r-r_1)^2 + Ch,
\end{equation} 
near $r_1$, and let  $w(r) =e^{-\alpha(r-r_1)^2/2h}\chi(r)$ where $\chi \in C_c^\infty((0,r_2);[0,1])$ is $1$ near $r_1$ and is supported inside the set where \eqref{e:vmalpha} holds. Then we have
\begin{equation}\label{e:specbd1}
0 \le \int(-h^2 w'' + V_mw - Ew) w \le (E_0-E + Ch)\int w^2,
\end{equation}
which concludes the proof that $|E - E_0| \le Ch$.

Let $u_0$ be a corresponding real-valued eigenfunction, and extend $u_0$ to solve \eqref{e:pjuj} on all of $\mathbb R_+$. 

Now we may again write $u_0$ in terms of Airy functions as in the proof of Lemma \ref{l:oneturn}, with $\zeta$ defined by \eqref{e:zetadef}, but now  $m' = m$ and $R$ are as in the statement of the present Lemma. The two main differences for our work here compared to that in the proof of Lemma \ref{l:oneturn} are that we have good remainder bounds only when $r \ge r_{1+}$, rather than for all $r \in (0,\infty)$, and that $m$ stays bounded.

More precisely, by \cite[\S 11.3.3]{Olver:Asymptotics}, there are real constants $A_0 = A_0(h)$ and $B_0 = B_0(h)$ such that for $r \ge r_{1+}$ we have \eqref{e:u0airy}, where $\varepsilon_A$ and $\varepsilon_B$ satisfy \eqref{e:epsab+} for $r \ge R$ and \eqref{e:epsab-} for  $r \in [r_{1+},R]$. We will need the following bounds on $\zeta$ near the turning point $R$:
\begin{equation}\label{e:zetarmid}
h|\zeta(r)|^{3/2}  = |V'_{m}(R)|^{1/2}|R-r|^{3/2}(1 + O(|R-r|)), \quad \textrm{ when } |r-R| \le 1,
\end{equation}
where we used $E - V_{m}(r) = (R-r)V_{m}'(R) + O((R-r)^2)$.

Without loss of generality we normalize $u_0$ so that
\begin{equation}\label{e:a02b02}
 A_0^2 + B_0^2 = 1, \qquad A_0 \ge 0.
\end{equation}
Since $u_0(r_2)=0$ we have
\[
 A_0\left(\Ai(-\zeta(r_2)) + O(h) \right) + B_0 \left(\Bi(-\zeta(r_2)) + O(h)\right)=0.
\]
Now observe that by $R = r_2 + O(h)$ and \eqref{e:zetarmid} we have $|\zeta(r_2)| \le C h^{1/3}$, and since $\Bi(0) = \Ai(0)\sqrt 3 >0$,  we obtain
\[
 A_0 +  B_0(\sqrt 3  + O(h^{1/3})) = 0,
\]
and combining with \eqref{e:a02b02} gives
\begin{equation}\label{e:a0b03}
 A_0 = \frac {\sqrt 3}2 +O(h^{1/3}), \qquad  B_0 = - \frac 12 +O(h^{1/3}).
\end{equation}
When $r \ge r_{2+}$ we have \eqref{e:u0lgm} with constants $A$ and $B$, which we can compute as in \eqref{e:apmib} to find
\begin{equation}\label{e:apmib2}
A \pm i B = 2^{-1}\pi^{-1/2}e^{\mp i \pi/4}(A_0 \pm i B_0) + O(h).
\end{equation}

We take $u_1$ to be given by \eqref{e:u1lgm} for $r\ge r_{2+}$ as before, where now \eqref{e:u1airy} holds for $r \ge r_{1+}$ with $A_1$ and $B_1$ given by \eqref{e:a1b1}. We now have \eqref{e:u1abs} for $r \ge r_{1+}$. This time the Wronskian $W=u_0u_1'-u_0'u_1$ obeys
\begin{equation}\label{e:w3}
 W= 2(B+iA)h^{-1} = e^{i11\pi/12} \pi^{-1/2}h^{-1}(1+O(h)),
\end{equation}
where we used \eqref{e:apmib2} and \eqref{e:a0b03}.

Inserting \eqref{e:ax} and \eqref{e:bx} into \eqref{e:u0airy} and \eqref{e:u1airy} and using \eqref{e:epsab-} and \eqref{e:zetarmid} gives
\begin{equation}\label{e:u0s}
\sqrt \pi u_j(r) = B_j (V_m(r)-E)^{-1/4} e^{S(r)/h}(1+O(h)) ,
\end{equation}
for $r \in [r_{1+}, r_{2-}]$, where $j \in \{0,1\}$. Then \eqref{e:krr'midmid} follows from inserting \eqref{e:w3} and \eqref{e:u0s} into \eqref{e:kdef} and using \eqref{e:a1b1} and \eqref{e:a0b03}. We obtain \eqref{e:krr'mid} by the same argument with \eqref{e:u1lgm} in place of \eqref{e:u0s} for $u_1$.
\end{proof}

\section{Proofs of Theorems}\label{s:proof2}

\subsection{Proof of Theorem \ref{t:second}}
Let $0 = \sigma_0 < \sigma_1 = \sigma_2 \le  \sigma_3 \le \cdots$ be the eigenvalues of the unit sphere of dimension $n-1$, repeated according to multiplicity, and let $Y_0,  Y_1, \ \dots$ be a corresponding sequence of orthonormal real eigenfunctions.

If $v$ is in the domain of $P$, then
\begin{equation}\label{e:vsep}
  r^{(n-1)/2}Pr^{-(n-1)/2} v(x) = \sum_{j=0}^\infty Y_jP_{m_j}v_j(r), \quad \textrm{ where } v(x) = \sum_{j=0}^\infty Y_jv_j(r),
\end{equation}
where $P_{m_j}$ is given by \eqref{e:radop} with
\begin{equation}\label{e:mjdef}
 m_j = m_j(h) =  h^2(\sigma_j + \tfrac{(n-1)(n-3)}4), \quad v_j(r) = \int Y_j(\theta)v(r,\theta)d\theta.\end{equation}
Here $d\theta$ is the usual surface measure on the unit sphere and $v_j \in \mathcal D$ for some $\mathcal D = \mathcal D(m_j,h)$ as in \S\ref{s:ode}.

Similarly, if $v \in L^2(\mathbb R^n)$ is compactly supported, then so is $v_j \in L^2(\mathbb R_+)$ and
\begin{equation}\label{e:limabssep}
  r^{(n-1)/2}(P-E \pm i 0)^{-1}r^{-(n-1)/2} v(x) = \sum_{j=0}^\infty Y_j(P_{m_j} - E \pm i0)^{-1}v_j(r),
\end{equation}
with the outgoing resolvent $(P_{m_j} - E - i0)^{-1}$ having integral kernel $K$ given by \eqref{e:kdef}, and the incoming resolvent  $(P_{m_j} - E + i0)^{-1}$ having integral kernel given by the complex conjugate of $K$. (Actually, in \eqref{e:limabssep} we could instead take $v$ in a weighted space but we will not need this.)

Hence if $\chi_L$ and $\chi_R$ are bounded and compactly supported functions on $[0,\infty)$, then
\[
 \|\chi_L(r)(P-E \pm i 0)^{-1}\chi_R(r)\|_{L^2(\mathbb R^n) \to L^2(\mathbb R^n)} = \sup_{j \in \mathbb N_0}  \|\chi_L(r)(P_{m_j}-E \pm i 0)^{-1}\chi_R(r)\|_{L^2(\mathbb R_+) \to L^2(\mathbb R_+)}.
\]

\begin{proof}[Proof of \eqref{e:texpasy}] Suppose without loss of generality that $r_L \le r_R$. Then \eqref{e:texpasy} follows from Lemma \ref{l:lower}, applied with $m= m_j$ for any $j = j(h)$ chosen such that $m_j = M_0 + O(h)$, and with $r_{1+}, \ r_{2-}, \ r_{2+}$ chosen such that $\partial B(0,r^*) \subset U_L$ for some $r^* \in  [r_{1+},r_{2-}]$ and $\partial B(0,r^{**}) \subset U_R$ for some  $r^{**} \in [r_{1+},r_{2-}] \cup [r_{2+},\infty)$. 
\end{proof}

\begin{proof}[Proof of \eqref{e:tnontr}] Fix $r_{2+}>r_2$ such that $U$ is disjoint from a neighborhood of $\overline{B(0,r_{2+})}$. Fix $M>M_0$ and $I \Subset (0,\infty)$ containing $E_0$ such that the hypotheses of  Lemmas \ref{l:noturn} and \ref{l:oneturn} are satisfied (it is enough if $M-M_0$ and the length of $I$ are sufficiently small). Then apply the Hilbert--Schmidt bound
\[
\|\chi(r)(P_{m_j}-E \pm i 0)^{-1}\chi(r)\|_{L^2(\mathbb R_+) \to L^2(\mathbb R_+)}^2 \le \|\chi\|_{L^\infty}^4\int_{\supp \chi} \int_{\supp \chi} |K(r,r')|^2drdr' \le Ch^{-2},
\]
which holds uniformly for all $j \in \mathbb N_0$.
\end{proof}

Finally, to prove \eqref{e:texpasyseq}, by Lemma \ref{l:lower} it suffices to show that $E_0$ is an eigenvalue of $P_{m_j}$ on $\mathcal{D}_{r_2}$, for some sequence $h_j \to 0$ such that $m_j = m_j(h_j)$ obeys $|m_j - M_0| \le C h_j$, for $j$ sufficiently large. This follows from a calculation very similar to that in \eqref{e:vmalpha} and \eqref{e:specbd1}, but note that now we will necessarily have $m_j\le M_0$; this is reasonable because we are now assuming $\max V < E_0$ and hence $M_0>0$ by definition \eqref{e:m0def}. We will first define the sequence $h_j$, then prove that
\begin{equation}\label{e:mjup}
 m_j \le M_0,
\end{equation}
and then prove that
\begin{equation}\label{e:mjlow}
 m_j \ge M_0 - C h_j.
\end{equation}

We define $h_j$ for $j$ sufficiently large by demanding that $h_j^{-2}$ be the bottom of the spectrum of
\[
L_j:= (E_0-V_0(r))^{-1}\left(-\partial_r^2 + (\sigma_j + \tfrac{(n-1)(n-3)}4)r^{-2}\right),
\]
as an operator on $L^2((0,r_2), (E_0 - V_0(r))dr)$ with domain 
\[\{u \in L^2((0,r_2)) \mid L_j u \in  L^2((0,r_2)) \textrm{ and } u(r_2) = 0\}.\]
Note that this operator is selfadjoint (that is, there is no need for an analogue of the condition $\mathcal B u = 0$ as in the definition of $\mathcal D$  in the beginning of \S\ref{s:ode}) as long as $4\sigma_j + (n-1)(n-3) > 3$, that is to say for all but possibly finitely many $j$: see \cite[Theorem 10.4.4]{zet} and also \cite[Theorem X.10]{rs}. The spectrum is discrete since the domain is contained $H^1((0,r_2))$. Also, $h_j \to 0$ will follow from \eqref{e:mjup}. Hence, to show \eqref{e:texpasyseq} it is enough to prove \eqref{e:mjup} and \eqref{e:mjlow}.

\begin{proof}[Proof of \eqref{e:mjup}]
We have, for $u \in C_c^\infty((0,r_2))$,
\[\begin{split}
 \int_0^{r_2} (L_j u)(r) \overline{u(r)} (E_0 - V_0(r))dr &\ge \int_0^{r_2}  (\sigma_j + \tfrac{(n-1)(n-3)}4)r^{-2}|u(r)|^2 dr \\ &\ge  M_0^{-1}(\sigma_j + \tfrac{(n-1)(n-3)}4) \int_0^{r_2}  |u(r)|^2(E_0 - V_0(r))dr, 
\end{split}\]
where we used the fact that by definition
\[
 M_0^{-1} = \min\{r^{-2}(E_0 - V_0(r))^{-1} \mid r \in (0,r_2)\}.
\]
This implies $ h_j^{-2} >  M_0^{-1}(\sigma_j + \tfrac{(n-1)(n-3)}4)$ and hence \eqref{e:mjup}.
\end{proof}

\begin{proof}[Proof of \eqref{e:mjlow}] Let $w(r) = e^{-\alpha(r-r_1)^2/2h_j} \chi(r)$, with $\alpha$ as in \eqref{e:vmalpha} and $\chi$ as in the line after. Then, as in \eqref{e:specbd1},
\[ \begin{split} 
0 &\le \int_0^{r_2} \left[(L_j - h_j^{-2}) w(r)\right]  w(r) (E_0 - V_0(r))dr 
\le  \int_0^{r_2}  \left( (\sigma_j  - M_0h_j^{-2})r^{-2}  + Ch_j^{-1}\right)w(r)^2  dr.
\end{split}\]
This implies $M_0 h_j^{-2} \le \sigma_j + C h_j^{-1}$ and hence \eqref{e:mjlow}.
\end{proof}

\subsection{Proof of Theorem \ref{t:first}}\label{s:4.2}

To deduce Theorem \ref{t:first} from \ref{t:second}, we begin with some useful formulas relating $M_0$, $r_1$, and $r_2$. We define
\[
 \Phi(r) = r^2(E_0 - V_0(r)).
\]
\begin{lem}\label{l:phi}
With the assumptions and notation of Theorem \ref{t:second}, we have 
\begin{equation}\label{e:phir2+}
 r\ge r_2 \Longrightarrow \Phi'(r)>0,
\end{equation}
\begin{equation}
 \Phi(r_1) = \Phi(r_2) = M_0,
\end{equation}
\begin{equation}
 r_1< r< r_2 \Longrightarrow \Phi(r)<M_0,
\end{equation}
\begin{equation}
 r \le r_2 \Longrightarrow \Phi(r)\le M_0.
\end{equation}
See Figure \ref{f:phi}.
\end{lem}

\begin{figure}[h]
\labellist
\small
\pinlabel $M_0$ [l] at 2 100
\pinlabel $r_1$ [l] at 90 16
\pinlabel $r_2$ [l] at 281 16
\endlabellist
 \includegraphics[width=10cm]{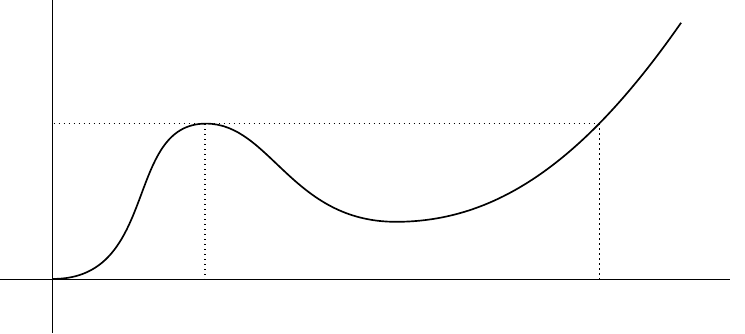}
 \caption{A possible graph of $\Phi$.}\label{f:phi}
\end{figure}

\begin{proof}
 To show \eqref{e:phir2+} we compute
 \[
  \Phi'(r) = 2r(E_0 - V_0(r)) - r^2 V_0'(r) =  2r(E_0-V_{M_0}(r)) - r^2 V_{M_0}'(r),
 \]
and the right hand side is positive when $r \ge r_2$ by \eqref{e:monass}. To prove the other three statements we observe $\Phi(r) \le M_0$ is the same as $V_{M_0(r)} \ge E_0$, with equality always holding simultaneously.
\end{proof}

Now suppose that $V_0$ is compactly supported in $[0,1)$ and $\min V_0<0$. Then $\Phi(1) = E_0$. If
\[
 E_0 < \max_{r \in [0,1)} (-r^2 V_0(r)) \le \max_{r \in [0,1)} \Phi(r) = M_0,
\]
then $r_1 < 1 < r_2$, and we have moreover
\begin{equation}\label{e:r2comp}
 r_2 = r_2(E_0) = \sqrt{M_0/E_0},
\end{equation}
and
\begin{equation}\label{e:r2asy}
 \lim_{E_0 \to 0} E_0 r_2^2 = \lim_{E_0 \to 0} M_0 =\min\{m>0 \mid V_m(r) \ge 0 \textrm{ for all }r>0\} = - \min \{r^2 V_0(r)\},
\end{equation}
as desired.

\subsection{Proof of Theorem \ref{t:wave}}\label{s:wave}

It is convenient to work over the Hilbert space $\mathcal H = \dot H^1(\mathbb R^n) \oplus L^2(\mathbb R^n, dx/c_0(r)^2)$. Let 
\[
 B = -i \left( \begin{array}{cc} 0 & 1 \\ c^2 \Delta & 0\end{array}\right).
\]
Then $B$ is selfadjoint on $\mathcal H$ with domain $\{(u_0,u_1) \in \mathcal H \colon \Delta u_0 \in L^2(\mathbb R^n), u_1 \in H^1(\mathbb R^n)\}$, and we will study the unitary wave propagator $e^{itB} \colon \mathcal H \to \mathcal H$.

Theorem \ref{t:wave} follows from

\begin{lem}\label{p:wave} Let $c, \, \rho,$ and $R_c$ be as in Theorem \ref{t:wave}.
 \begin{enumerate}
  \item If $\chi \in C_c^\infty(\mathbb R^n)$ has support disjoint from the closed ball $\overline{B(0,R_c)}$, then there is $C>0$ such that
\begin{equation}\label{e:wdecay}
 \int_{-\infty}^\infty \|\chi e^{itB} u\|_{\mathcal H}^2dt \le C \|u\|_{\mathcal H}^2,
\end{equation}
for all $u \in \mathcal H$.
\item If $\chi \in C_c^\infty(\mathbb R^n;[0,\infty))$ is positive on the sphere $\partial B(0,R)$ for some $R \in [\rho ,R_c]$, then
\begin{equation}\label{e:wnondecay}
\sup_{\|u\|_{\mathcal H} = 1} \int_{-\infty}^\infty \|\chi e^{itB} u\|_{\mathcal H}^2dt = +\infty.
\end{equation}
 \end{enumerate}
\end{lem}
Indeed, to prove Theorem \ref{t:wave} from Lemma \ref{p:wave}, we observe that if $\supp \chi_0 \subset U$ and $\chi_1 = 1$ near $U$, then
\[
C^{-1}\|\chi_0 e^{itB} u\|_{\mathcal H}^2 \le \mathcal E_U(t) \le C \|\chi_1 e^{itB} u\|_{\mathcal H}^2,
\]
with $u = (w,\partial_tw)$, where for the first inequality we used Poincar\'e's inequality 
\[\|v\|_{L^2(U)} \le C \|\nabla v\|_{L^2(U)}, \qquad v \in C_c^\infty(U).\]

To prove Lemma \ref{p:wave}, we will need some facts about the resolvent of $B$, based on the formula
\begin{equation}\label{e:resbblock}
 (B- \lambda)^{-1} = \left(\begin{array}{cc} \lambda  (- c^2\Delta - \lambda^2)^{-1}  \quad & -i  (- c^2\Delta - \lambda^2)^{-1} \\ i\lambda^2(- c^2\Delta - \lambda^2)^{-1} + i \quad & \lambda  (- c^2\Delta - \lambda^2)^{-1} \end{array} \right), \quad \lambda \in \mathbb C \setminus \mathbb R;
\end{equation}
see  \cite[p. 265]{pv}, \cite[(2.13)]{burq2003}, and \cite[(6.6)]{sh2}. For any $\chi \in C_c^\infty(\mathbb R^n)$, the cutoff resolvent $\chi (- c^2\Delta - \lambda^2)^{-1} \chi$ extends continuously from the lower, or upper, half plane to its closure, as an operator from $L^2(\mathbb R^n)$ to $H^2(\mathbb R^n)$ (see item 1 of \cite[Lemma 4.1]{ddh} for a proof using the Sj\"ostrand--Zworski black box theory \cite{sz}). By \eqref{e:resbblock} we see that $\chi (B-\lambda)^{-1}\chi$ has corresponding continuous extensions as an operator from $\mathcal H$ to $\mathcal H$, and we denote these by $\chi (B - \lambda \pm i0)^{-1} \chi$, where $\lambda \in \mathbb R$.

\begin{lem}\label{l:bres}
Let $c, \, \rho,$ and $R_c$ be as in Theorem \ref{t:wave}.
 \begin{enumerate}
  \item If $\chi \in C_c^\infty(\mathbb R^n)$ has support disjoint from the closed ball $\overline{B(0,R_c)}$, then there is $C>0$ such that
\begin{equation}\label{e:wnontr}
\|  \chi (B - \lambda \pm i0)^{-1} \chi \|_{\mathcal H \to \mathcal H} \le C,
\end{equation}
for all $\lambda \in\mathbb R$.

\item If $\chi \in C_c^\infty(\mathbb R^n;[0,\infty))$ is positive on the sphere $\partial B(0,R)$ for some $R \in [\rho ,R_c]$, then there are $C>0$ and a sequence $\lambda_j \to + \infty$ such that
 \begin{equation}\label{e:wexp}
\| \chi \left[(B - \lambda_j + i0)^{-1} - (B - \lambda_j - i0)^{-1}\right] \chi \|_{\mathcal H \to \mathcal H} \ge e^{C\lambda_j}.
\end{equation}
\end{enumerate}
\end{lem}

\begin{proof}[Proof of Lemma \ref{l:bres}]
 We use the identity
 \[
 (- c^2\Delta - (\lambda \pm i0)^2)^{-1} = (-h^2 \Delta + V - (\kappa^{-2} \pm i0))^{-1} c^{-2}\lambda^{-2},
\]
where $h = \lambda^{-1}$ and $V = \kappa^{-2} - c^{-2}$. To check this, note that if $f \in L^2(\mathbb R^n)$ is compactly supported, and
\[
 (-c^2\Delta - \lambda^2)u = f, \qquad  (-h^2 \Delta + V - \kappa^{-2})v = c^{-2}\lambda^{-2}f,
\]
with $u$ and $v$ both outgoing, then $u=v$ (see \cite[Theorem 3.34 and Theorem 4.18]{dz}). 

With $E_0 = \kappa^{-2}$ and $V_0 = \kappa^{-2} - c_0^{-2}$, we have, in the notation of Theorem \ref{t:second} and Lemma \ref{l:phi}, $\Phi(r) = r^2/c_0^2(r)$. Then
\[
M_0= \max_{r \in [0,\rho]}\Phi(r) =R_c^2/\kappa^2 > \rho^2/\kappa^2 = \Phi(\rho),
\]
and hence $r_1 < \rho < r_2 = R_c$.

\begin{enumerate}
 \item By \eqref{e:tnontr}, for all $\chi \in C_c^\infty(\mathbb R^n)$ having support disjoint from the closed ball $\overline{B(0,R_c)}$,  we have
 \[
   \| \chi (- c^2\Delta - (\lambda \pm i0)^2)^{-1} \chi \|_{L^2 \to L^2}  \le C\langle \lambda\rangle^{-1}. 
 \]
 By a standard argument (see for example the proof of \cite[Proposition 2.4]{burq2003}), together with \eqref{e:resbblock} this implies
\eqref{e:wnontr} for all such $\chi$.

\item To prove \eqref{e:wexp}, we argue similarly, but using the following refined version of \eqref{e:texpasyseq}:
\[
 \|\chi \left[ (-h_j^2 \Delta + V - (\kappa^{-2} + i0))^{-1} - (-h_j^2 \Delta + V - (\kappa^{-2} - i0))^{-1}\right]\chi\|_{L^2(\mathbb R^n) \to L^2(\mathbb R^n)} \ge e^{C/h_j},
\]
where $\chi \in C_c^\infty(\mathbb R^n;[0,\infty))$ is positive on the sphere $\partial B(0,R)$ for some $R \in [\rho ,R_c]$, and $h_j$ is the same sequence appearing in \eqref{e:texpasyseq}. This refined version follows from \eqref{e:krr'midmid} in the same way that \eqref{e:texpasyseq} does, and it implies \eqref{e:wexp} with $\lambda_j = h_j^{-1}$.
\end{enumerate}
\end{proof}

\begin{proof}[Proof of Lemma \ref{p:wave}]

This is a version of Kato smoothing \cite{k}; see also \cite[\S XIII.7]{rs} for another general presentation of the theory. We will use an $AA^\ast$ argument as in \cite[\S 2.3]{bgt} and \cite[\S 2]{burq2003}; see also \cite[\S 7.1]{dz}. We define the operator
$$A:\mathcal{H} \ni u\mapsto \chi e^{itB}u\in \mathcal S'(\mathbb{R};\mathcal{H})$$
with adjoint
$$A^\ast \colon  \mathcal S(\mathbb{R} ;\mathcal{H}) \ni f \mapsto\int_\mathbb{R}\chi 
e^{-isB}f(s)ds \in \mathcal H.$$
Boundedness of $A$ from $\mathcal H$ to $L^2(\mathbb R; \mathcal H)$ is equivalent to boundedness of $AA^\ast$ from $L^2(\mathbb R; \mathcal H)$ to $L^2(\mathbb R; \mathcal H)$  . We write
\begin{equation}\label{e:aa*def}
AA^\ast f= \chi \int_{-\infty}^t e^{i(t-s)B}\chi f(s)ds + \chi \int_t^\infty  e^{i(t-s)B}\chi f(s)ds =: \chi u_-(t) + \chi u_+(t).
\end{equation}
Now let us suppose temporarily that there is $T>0$ such that
\begin{equation}\label{e:fcomp}
 \supp f \subset [-T,T],
\end{equation}
so that in particular
\[
 \supp u_- \subset [-T,\infty) \textrm{ and } \supp u_+ \subset (-\infty,T].
\]
We use the equations
\[
 u'_\pm(t) = i B u_\pm(t) \mp \chi f(t)
\]
to compute the Fourier--Laplace transforms
\begin{equation}\label{e:upm}
\hat u_\pm (\lambda \pm i \varepsilon) : = \int_{\mathbb R} e^{-i(\lambda \pm i \varepsilon)t}u_\pm (t)dt = \mp i (B- (\lambda \pm i \varepsilon))^{-1} \chi \hat f(\lambda \pm i \varepsilon), 
\end{equation}
where $\lambda \in \mathbb R$ and $\varepsilon >0$.
By Plancherel's theorem,
\begin{equation}\label{e:plan}\begin{split}
 \int_{\mathbb R} \|A A^* f\|^2_{\mathcal H} dt &= \frac 1 {2\pi} \int_{\mathbb R} \|\chi \hat u_-(\lambda-i0) + \chi \hat u_+(\lambda+i0)\|^2_{\mathcal H} d\lambda\\
 &= \frac 1 {2\pi} \int_{\mathbb R} \left\|\chi \left[ (B - \lambda + i0)^{-1} - (B-\lambda - i0)^{-1}\right]\chi \hat f(\lambda) \right\|^2_{\mathcal H} d\lambda,
\end{split}\end{equation}
where we used \eqref{e:aa*def} and \eqref{e:upm}, and where we allow the integrals to take the value $+\infty$.
By density, \eqref{e:plan} holds also without the assumption \eqref{e:fcomp}. Now \eqref{e:wdecay} follows from \eqref{e:wnontr}.

To deduce \eqref{e:wnondecay} from \eqref{e:wexp}, we must show that \eqref{e:wexp} implies
\[
 \left\|\chi \left[ (B - \bullet + i0)^{-1} - (B-\bullet - i0)^{-1}\right]\chi \right\|_{L^2(\mathbb R; \mathcal H) \to L^2(\mathbb R; \mathcal H)} = + \infty.
\]
But this is clear since $\lambda \mapsto \chi \left[ (B - \lambda + i0)^{-1} - (B-\lambda - i0)^{-1}\right]\chi$ is continuous by the discussion following \eqref{e:resbblock} and unbounded by \eqref{e:wexp}.




\end{proof}

\appendix

\section{Airy functions}
\label{s:airy}
In this appendix we review some needed facts about the Airy functions given by
\begin{equation}\label{e:airydef}\begin{split}
 \Ai(x) &= \frac 1 \pi \int_0^\infty \cos\left(\tfrac {t^3}3 + xt\right)dt, \\
 \Bi(x) &= \frac 1 \pi \int_0^\infty \left(e^{-\frac{t^3}3 + xt} +\sin\left(\tfrac {t^3}3 + xt\right)\right)dt.
\end{split}\end{equation}
These are  solutions to  $u'' = xu$ satisfying $\Bi(0) = \sqrt 3 \Ai(0)>0$,
and as $x \to \infty$ we have
\begin{equation}\label{e:ax}
2 \sqrt \pi x^{1/4} \Ai(x) = e^{- 2x^{3/2}/3}(1 + O(x^{-3/2})),
\end{equation}
\begin{equation}\label{e:bx}
 \sqrt \pi x^{1/4} \Bi(x) = e^{ 2x^{3/2}/3}(1 + O(x^{-3/2})),
\end{equation}
\begin{equation}\label{e:a-x}
\sqrt \pi x^{1/4} \Ai(-x) = \cos\left(\tfrac 23 x^{3/2} - \tfrac \pi 4\right) + O(x^{-3/2}),
\end{equation}
\begin{equation}\label{e:b-x}
-\sqrt \pi x^{1/4} \Bi(-x) = \sin\left(\tfrac 23 x^{3/2} - \tfrac \pi 4\right) + O(x^{-3/2}).
\end{equation}
These results can be found in \cite[\S11.1]{Olver:Asymptotics}, among other places. In particular,  there is a constant $C_A>0$ such that, for all real $x$ and $x'$ satisfying $x' \le x$, we have
\begin{equation}\label{e:abbound}
 |x^{1/2}{x'}^{1/2}\Ai(x)^2(\Ai(x')^2 + \Bi(x')^2)| \le C_A^2.
\end{equation}

\section{Remainder bounds for Airy approximations}\label{s:apperror}

In this appendix we prove \eqref{e:epsab+} and \eqref{e:epsab-}. By \cite[Chapter 11, Theorem 3.1]{Olver:Asymptotics}, it is enough to show that $\int_0^\infty |G(r)|dr$ is uniformly bounded for all $h$ and $m$, where
\begin{equation}\label{e:gabs}
 G(r) = |f(r)|^{-1/2}\left(5 f(r)^{-2}f'(r)^2 - 4 f(r)^{-1}f''(r) - 16 g(r) - 5f(r)\zeta_O(r)^{-3}\right).
\end{equation}
with $f$ and $g$ as in \eqref{e:fgdef} and
\begin{equation}\label{e:zetao}
 \zeta_O(r) = - m^{-1/3}h^{2/3}\zeta(r) = \pm \left|\frac 32 \int_R^r\sqrt{f(r')}dr'\right|^{2/3}, \textrm{ when } \pm(R-r)\ge 0.
\end{equation}
 We will first show that there is $\delta> 0$ such that
\begin{equation}\label{e:gmidbound}
  \int_{R-\delta\sqrt m}^{R+\delta \sqrt m}|G(r)|dr \le C,
\end{equation}
and then we will show that
\begin{equation}\label{e:gnonmidbound}
  \int_{R+\delta\sqrt m}^\infty|G(r)|dr  +   \int_0^{R-\delta\sqrt m}|G(r)|dr\le C.
\end{equation}

\begin{proof}[Proof of \eqref{e:gmidbound}]
By Taylor's theorem, for $r \in [R-\delta \sqrt m, R+\delta \sqrt m]$ we have
\[
 f(r) = (r-R)f'(R)\left(1 + \frac {f''(R)}{2f'(R)}(r-R) + O(m^{-1}(r-R)^2)\right),
\]
where we used $f'(R) \le - m^{-3/2}/C$ and $|f'''(r)| \le C m^{-5/2}$  (these follow from \eqref{e:sqrtm}).
Similar expansions hold for powers and derivatives of $f$, and inserting the expansion for $\sqrt f$ into \eqref{e:zetao} gives
\[
 \zeta_O(r) = (r-R)f'(R)^{1/3}\left(1 +  \frac {f''(R)}{10f'(R)}(r-R) + O(m^{-1}(r-R)^2)\right).
\]
Hence
\[
 f(r)^{-2}f'(r)^2  = (r-R)^{-2}\left(1 + \frac {f''(R)}{f'(R)}(r-R) + O(m^{-1}(r-R)^2) \right),
\]
\[
 f(r)^{-1}f''(r) = (r-R)^{-1}\frac{f''(R)}{f'(R)}\left(1 +  O(m^{-1}(r-R)^2)\right),
\]
\[
 f(r)\zeta_O(r)^{-3} = (r-R)^{-2}\left(1 + \frac {f''(R)}{5f'(R)}(r-R) +  O(m^{-1}(r-R)^2)\right).
\]
Combining these and using $|g(r)| \le C m^{-1/2}$ gives
\[
 |G(r)| \le Cm^{-1}(r-R)^{-1/2},
\]
which implies \eqref{e:gmidbound}.
\end{proof}



\begin{proof}[Proof of \eqref{e:gnonmidbound}]

To bound the first term in \eqref{e:gnonmidbound} we use
\[
m|f(r)| + m^{-1}|f(r)^{-1}|+ r^2|g(r)| + r^3|f'(r)| + r^4|f''(r)| \le C,
\]
and
\begin{equation}\label{e:zetarlarge}
 \left|\tfrac 23 m^{1/2} \zeta_O(r)^{3/2}  -r\sqrt {E}\right| =   \left|\int_R^r \frac{V_{m'}(r')dr'}{ (\sqrt{E-V_{m'}(r')} + \sqrt E)}  +R\sqrt {E}\right|  \le C,
\end{equation} 
for $r \ge R + \delta \sqrt m$. These imply
\[
 |G(r)| \le C m^{1/2}r^{-2},
\]
for $r \ge R + \delta \sqrt m$, which implies the bound on the first term in \eqref{e:gnonmidbound}.

To bound the second term in \eqref{e:gnonmidbound}, we use
\[
r^2f(r) + 
 \ln^{-2}(R/r)\zeta_O(r)^{-3} \le C,
\]
for $r \le R-\delta  \sqrt m$ (for the last term c.f. \eqref{e:zetarsmall}), which implies
\begin{equation}\label{e:gsmall1}
\int_0^{R-\delta  \sqrt m}f(r)^{1/2} \zeta_O(r)^{-3}dr \le C.
\end{equation}
Let $F = \lim_{r \to 0}r^2f$, so that for $k \in \{0, \, 1, \, 2\}$ we have
\[
 \partial_r^k f(r) = (-1)^k(k+1)!F r^{-2 - k}(1 + O(m^{-1/2}r)),
\]
for $r\le R - \delta \sqrt m$. Then, since $F \ge 1/C$, we have
\[
f(r)^{-1/2}\left|5 f(r)^{-2}f'(r)^2  - 4f(r)^{-1} f''(r) - 16g(r)\right| \le C m^{-1/2}, 
\]
and inserting into \eqref{e:gabs} and combining with \eqref{e:gsmall1} gives the bound on the second term in \eqref{e:gnonmidbound}.
\end{proof}

\end{document}